\documentclass[12pt,bezier]{article}
\usepackage{amssymb}
\usepackage{mathrsfs}
\usepackage{amsmath}
\usepackage{amsfonts,amsthm,amssymb}
\usepackage{amsfonts}
\usepackage{graphics}
\usepackage{cite}

\newtheorem{lem}{Lemma}[section]
\newtheorem{thm}[lem]{Theorem}

\newtheorem{rem}[lem]{Remark}

\begin{document}

\title{Reliability evaluation of folded hypercubes in terms of component connectivity
%\footnote{The research is supported by  NSFC (No.11301371, 11301217).}
}
\author{ Shuli Zhao, \quad Weihua Yang\footnote{Corresponding author. E-mail: ywh222@163.com,~yangweihua@tyut.edu.cn}\\
\\ \small Department of Mathematics, Taiyuan University of Technology,\\
\small  Taiyuan Shanxi-030024,
China}
\date{}
\maketitle

{\small{\bf Abstract.}\quad The component connectivity is the generalization of connectivity which is an parameter for the reliability
evaluation of interconnection networks. The $g$-component connectivity $c\kappa_{g}(G)$ of a non-complete connected graph $G$ is the minimum
number of vertices whose deletion results in a graph with at least $g$ components. The results in [Component connectivity of the hypercubes, International Journal of Computer Mathematics 89 (2012)  137-145] by Hsu et al. determines the component connectivity of the hypercubes. As an invariant of the hypercube, we determine the $(g+1)$-component connectivity
of the folded hypercube $c\kappa_{g}(FQ_{n})=g(n+1)-\frac{1}{2}g(g+1)+1$ for $1\leq g \leq n+1, n\geq 8$ in this paper.
\section{Introduction}
An interconnection network is usually modeled by a connected graph in which vertices represent processors and edges represent links between processors. The connectivity is one of the important parameters to evaluate the reliability and fault tolerance of a network. However, the traditional connectivity always underestimates the resilience of large networks. With the development of multiprocessor systems, improving the traditional connectivity is necessary. Motivated by the shortcomings of traditional connectivity, several generalized connectivity and edge connectivity were considered by many authors. The component connectivity and component edge connectivity were introduced in \cite{Chartrand} and \cite{Sampathkumar} independently. The component connectivity can more accurately evaluate the reliability and fault tolerance for large-scale parallel processing systems accordingly. In \cite{Hsu}, it studies the component connectivity of the hypercubes.

Let $G$ be a non-complete connected graph. The $g$-component cut of $G$ is a set of vertices whose deletion results in a graph with at least $g$ components. The $g$-component connectivity $c\kappa_{g}(G)$ of a graph $G$ is the size of the smallest  $g$-component cut of $G$. By the definition of $c\kappa_{g}(G)$, it can be seen that $ c\kappa_{g+1}(G)\geq c\kappa_{g}(G)$ for every positive integer $g$.

%\scalebox{0.9}{\includegraphics{fig1.eps}}
%Figure 1. The $3-$dimensional folded hypercube.

%\begin{center}
%\scalebox{0.9}{\includegraphics{fig2.eps}}
%Figure 2. The $4-$dimensional folded hypercube.
The $n$-dimensional hypercube $Q_{n}$ is an undirected graph $Q_{n}=(V,E)$ with $|V|=2^{n}$ and $|Q_{n}|=n2^{n-1}.$ Each vertex can be represented by an $n$-bit binary string, and every  bit position is $0$ or $1$. There is an edge between two vertices whenever there binary string representation differs in only one bit position.

As one of the important variants of the hypercube network, the $n$-dimensional folded hypercube $FQ_{n}$, proposed by El-Amawy \cite{El-Amawy}, is obtained from an $n$-dimensional hypercube $Q_{n}$ by adding an edge between any pair of vertices with complementary addresses. The folded hypercube $FQ_{n}$ is superior to $Q_{n}$ in some properties, see \cite{El-Amawy, Lai,Chang}. Thus the folded hypercube is an enhancement on the hypercube $Q_{n}$ and $FQ_{n}$ is obtained by adding a perfect matching $M$ on the hypercube, where $M=\{(u,\overline{u})|u\in V(Q_{n})\}$ and $\overline{u}$ represents the complement of the vertex $u,$ that is, all their binary strings are complement and $\overline{0}=1$ and $\overline{1}=0$. One can seen that $E(FQ_{n})=E(Q_{n})\bigcup M.$ For convenience, $FQ_{n}$ can be expressed as $D_{0}\bigotimes D_{1},$ where $D_{0}$ and $D_{1}$ are $(n-1)$-dimensional subcubes induced by the vertices with the $i$-th coordinate $0$ and $1$ respectively.  The 3-dimensional and 4-dimensional folded hypercubes are shown in the
following Figure 1 and Figure 2,  respectively.
\begin{center}
\scalebox{0.15}{\includegraphics{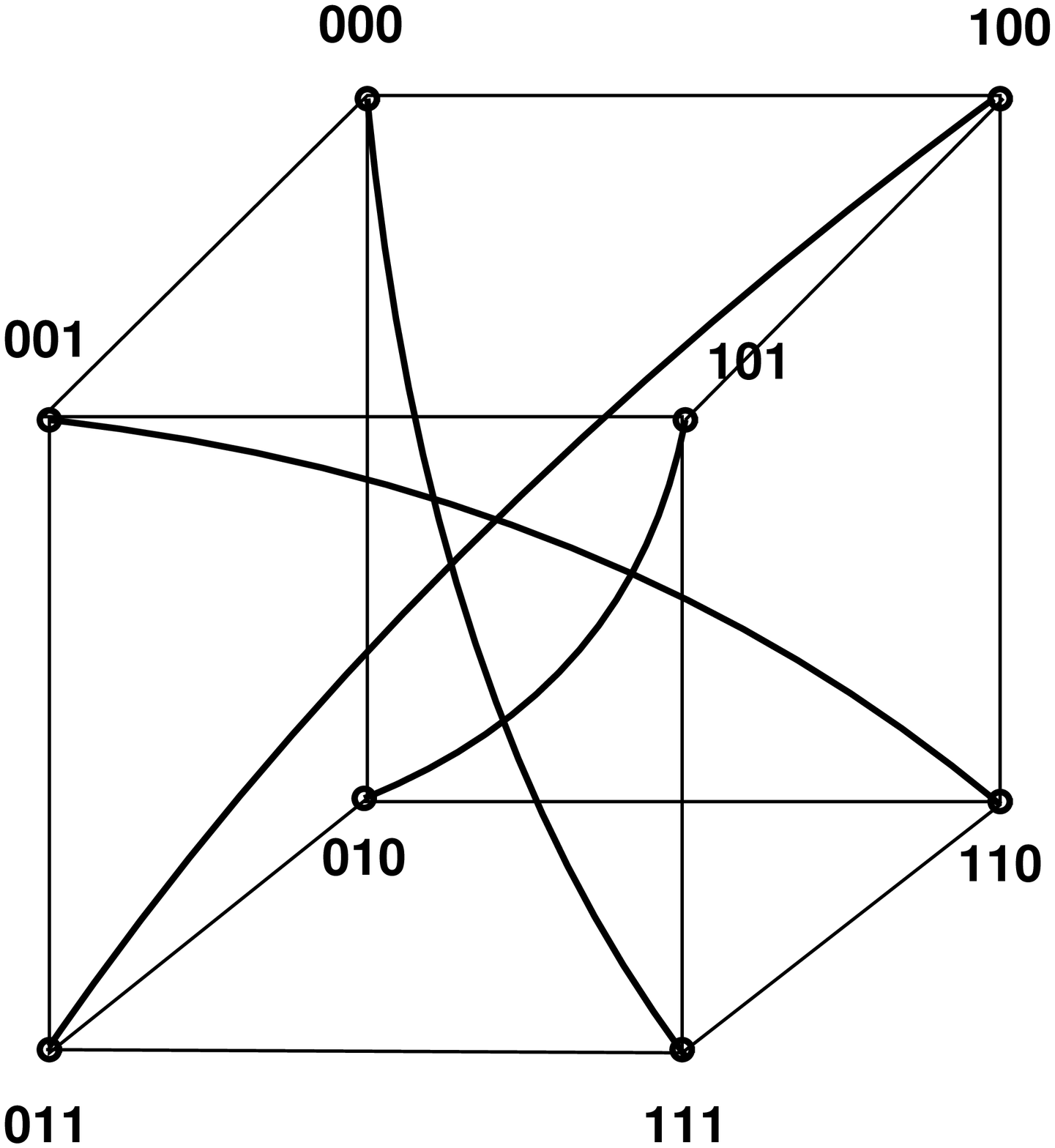}}\\
Figure 1. The 3-dimensional Folded hypercube.
\end{center}

\begin{center}
\scalebox{0.3}{\includegraphics{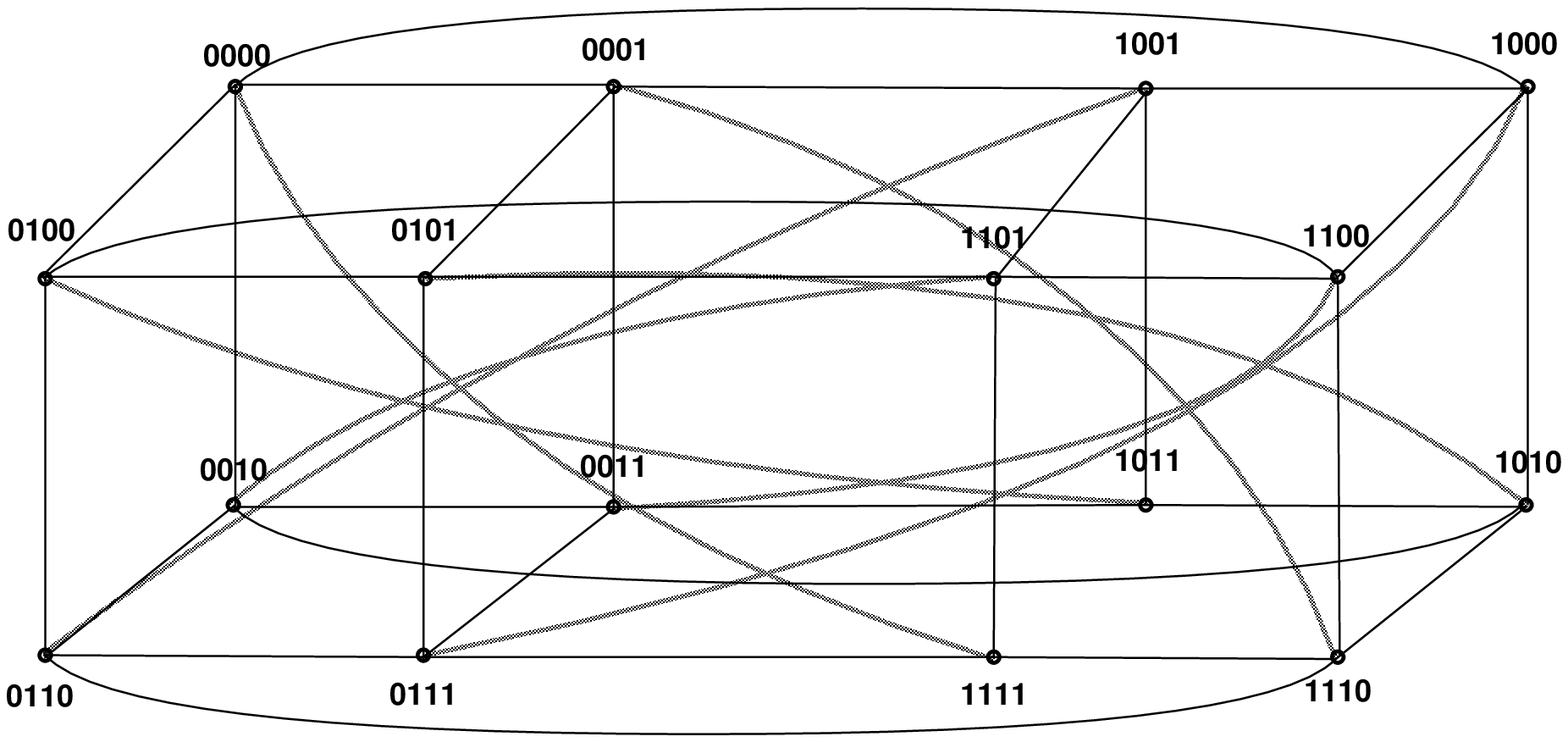}}\\
Figure 2. The 4-dimensional Folded hypercube.
\end{center}

Let $v$ be a vertex of a graph $G$, we use $ N_{G}(v)$ to denote the vertices that are adjacent to $ v $. Let $u,v \in V(G),$ $d(u,v)$ denotes the distance between $u$ and $v$. Let $ A\subseteq V(G),$ we denote by $ N_{G}(A)$ the vertex set $ \bigcup _{v \in V(A) }N_{G}(v)\setminus V(A)$  and $ C_{G}(A)=  N_{G}(A)\bigcup A.$ And also, we use $\theta_{G}(g)$ denotes the minimum number of vertices that are adjacent to a vertex set with $g$ vertices in $G$. For a vertex $v\in V(G)$, $d(v)$ denotes the degree of the vertex $v$ in $G$. The private neighbours of a vertex $v\in V^{'} \subseteq V(G),$ as in \cite{Somani}, denoted by $PN(v),$ are those neighbours of $v$ which are not shared by other vertices in $V^{'}$ and are not themselves in $V^{'}$, i.e., $PN(v)=N(v)-N(V^{'}-\{v\})-V^{'}.$ We follow Bondy \cite{Bondy} for terminologies not given here.

\section{Main results}

Before discussing $c\kappa_{g}(FQ_{n}),$ we need some results of hypercubes $Q_{n}$. Based on the results of  \cite{Somani, Yang2}, we have the following.

\begin{lem}[\cite{Somani, Yang2}]

$\theta_{Q_{n}}(g)= -\frac{1}{2}g^{2} +(n-\frac{1}{2})g+1$ for $1\leq g \leq n+1$,

$\theta_{Q_{n}}(g)= -\frac{1}{2}g^{2} +(2n-\frac{3}{2})g-n^{2}+2$ for $n+2\leq g \leq 2n.$

\end{lem}

\begin{lem}[\cite{zhao, Hsu}]
\[c\kappa_{g+1}(Q_{n})=\left\{
\begin{array}{ll}
          -\frac{1}{2}g^{2} +(n-\frac{1}{2})g+1,\hspace{2.8cm} 1\leq g \leq n, n\geq3 \\
-\frac{g^{2}}{2}+(2n-\frac{5}{2})g-n^{2}+2n+1, \hspace{1cm} n+1\leq g \leq 2n-5, n\geq6 \\
\end{array}
\right.\]
\end{lem}

\begin{lem}[\cite{Yang3}] Let $n\geq 4$ and $F\subseteq V(Q_{n}).$ Then the following holds.

$(i)$ If $|F|< \theta_{Q_{n}}(g)$ and $1\leq g \leq n-3,$ then $Q_{n}-F$ contains exactly one large component of order at least $2^{n}-|F|-(g-1).$

$(ii)$ If $|F|< \theta_{Q_{n}}(g)$ and $n-2\leq g \leq n+1,$ then $Q_{n}-F$ contains exactly one large component of order at least $2^{n}-|F|-(n+1).$

$(iii)$If $|F|< \theta_{Q_{n}}(g)$ and $n+2\leq g \leq 2n-4,$ then $Q_{n}-F$ contains exactly one large component of order at least $2^{n}-|F|-(g-1).$
\end{lem}

To determine $c\kappa_{g}(FQ_{n}),$ we need several properties of folded hypercubes.

\begin{lem}[\cite{Zhu}] Any two vertices in $V(FQ_{n})$ exactly have two common neighbours for $n\geq 4$ if they have.

\end{lem}

\begin{lem}\label{2.4} For any two vertices $v_{i}, v_{j}\in V(FQ_{n}), n\geq 4$, the following holds.

$(i)$ If $ d(v_{i}, v_{j})\neq 2,$ then these two vertices do not have any common neighbour.

$(ii)$ If $ d(v_{i}, v_{j})= 2,$ then these two vertices have exactly two common neighbours.
\end{lem}
\begin{proof} By the definition, the result holds.
\end{proof}
\begin{lem} Let $FQ_{n}$ be a folded hypercube of dimension $n,$ $n\geq 4,$ and $V'\subset V$ with $|V'|=2.$ Then $|N_{FQ_{n}}(V')|\geq 2n.$
\end{lem}
\begin{proof} In a folded hypercube of dimension $n,$ each vertex has $n+1$ neighbours. Let $V'=\{u, v\}$, then from Lemma \ref{2.4} they have at most two common neighbours, hence $|N_{FQ_{n}}(V')|\geq 2(n+1)-2=2n.$
\end{proof}
\begin{lem} Assume $n\geq 5$ and $v_{i}, v_{j}, v_{k} \in V(FQ_{n})$ such that $d(v_{i},v_{j})= d(v_{i},v_{k})= d(v_{j},v_{k})= 2.$ Then the three vertices have exactly four vertices which are common neighbours of at least two of them.
\end{lem}
\begin{proof} Let $v_{i}, v_{j}, v_{k} \in V$, such that $d(v_{i},v_{j})= d(v_{i},v_{k})= d(v_{j},v_{k})= 2.$ Let $b_{i}, b_{j}, b_{k}$ denote the corresponding address of these vertices. Because of the symmetric of folded hypercube, we can assume, without loss of generality, that $b_{i}=0\cdots 000\cdots 000\cdots 0, b_{j}=0\cdots 01_{j_1}0\cdots 01_{j_2}0\cdots 0$ and $b_{k}=0\cdots 1_{j_3}00\cdots 01_{j_2}0\cdots 0$, where $ 1_{j_l}$ means the $j_l-th$ binary bit be $1$, $1\leq j \leq n$. The neighbours to at least two of the vertices $v_{i}, v_{j}, v_{k}$ are the nodes with the following addresses: $b_{1}=0\cdots 1_{j_1}00\cdots 000\cdots 0$, $b_{2}=0\cdots 000\cdots 01_{j_2}0\cdots 0$, $b_{3}=0\cdots 01_{j_3}0\cdots 000\cdots 0$ and $b_{4}=0\cdots 1_{j_3}1_{j_1}0\cdots 01_{j_2}0\cdots 0$. This proves the result.
\end{proof}
\begin{lem}\cite{Xu} $FQ_{n}$ is a bipartite graph if and only if $n$ is odd.
\end{lem}
\begin{lem}\cite{Xu} If $FQ_{n}$ contains an odd cycle, then any shortest odd cycle contains exactly one complementary edge and the length is $n+1.$
\end{lem}
\begin{lem} Assume $n\geq 5$. If $V^{'}\subseteq V(FQ_{n})$ and $|V^{'}|=g,$ $1\leq g \leq n+1,$ then there exists at least one vertex $v\in V^{'}$ such that one of the following hold:

$(i)$ $v$ has a neighbour in $V^{'}$ and $|PN(v)|\geq n-g+2,$

$(ii)$ $v$ has no neighbours in $V^{'}$ and $|PN(v)|\geq n-g+1.$
 \end{lem}
\begin{proof}
We prove the result by considering the following two cases.

{Case 1.} There exist $v_{1}, v_{2}\in V^{'}$, such that $d(v_{1}, v_{2})=1.$

According to Lemma $2.4,$ these two vertices do not share any common neighbour. Hence, they have $2(n+1)-2= 2n$ neighbours that are not shared between them. From Lemma $2.7$ and Lemma $2.8,$ for $n\geq 5,$ $FQ_{n}$ has no odd cycle of length $5.$ Then each additional vertices $v_{j}\in V^{'} (j=3,4,\cdots\cdots g)$ can be in one of the following positions in relation to $v_{1}$ and $v_{2}.$

$(i)~ d(v_{1},v_{j})=1$ and $d(v_{2},v_{j})=2.$ Then by Lemma $2.4$, $v_{j}$ has no common neighbour with $v_{1}$ and has one additional common neighbour with $v_{2}$, the other common neighbour is $v_{1}$.

$(ii)~ d(v_{1},v_{j})=2$ and $d(v_{2},v_{j})> 2.$ Then by Lemma $2.4$, $v_{j}$ has two common neighbours with $v_{1}$ and no common neighbour with $v_{2}.$

$(iii)~ d(v_{1},v_{j})>2$ and $d(v_{2},v_{j})> 2.$ Then by Lemma $2.4$, $v_{j}$ has no common neighbours with either $v_{1}$ or $v_{2}.$

By Lemma $2.7$ and Lemma $2.8$ for $n\geq 5,$ we need not to consider the case $d(v_{1},v_{j})=2$ and $d(v_{2},v_{j})=2.$

Thus, for every $v_{j}\in V^{'}, j=3,4,\cdots g$ there are at most two vertices that are common neighbours to $v_{j}$ and either with $v_{1}$ or $v_{2}$. This implies that $|PN(v_{1})|+|PN(v_{2})|\geq 2n-2(g-2)= 2(n-g+2).$ Hence, either $|PN(v_{1})|\geq n-g+2 $ or $|PN(v_{2})|\geq n-g+2 $ or both of them hold.

{Case 2.} There exist no $v_{i}, v_{j} \in V^{'}$ such that $d(v_{i}, v_{j})=1$. The result can be in the following cases.

{Subcase 2.1}. There exist $v_{1}, v_{2} \in V^{'}$ such that $d(v_{1},v_{2})=2.$

 By Lemma $2.5$, the two vertices have two common neighbours. Hence, they have $2(n+1)-2\times 2 = 2(n-1)$ neighbours that are not shared between them. Each additional vertices $v_{j}\in V^{'} (j=3,4,\cdots\cdots g)$ can be in one of the following positions in relation to $v_{1}$ and $v_{2}.$

$(i)~ d(v_{1},v_{j})= d(v_{2},v_{j})=2$, then $d(v_{1},v_{2})= d(v_{1},v_{j})= d(v_{2},v_{j})= 2.$ By Lemma $2.6$, there are only four vertices which are common neighbours to at least two of them. As $v_{1}$ and $v_{2}$ have two common neighbours, so $v_{j}$ has at most two common neighbours with either $v_{1}$ or $v_{2}$.

$(ii)~ d(v_{1},v_{j})= 2$ and $d(v_{2},v_{j})> 2$, then by Lemma $2.4$, $v_{j}$ has two common neighbours with $v_{1}$ and no common neighbour with $v_{2}.$

$(iii)~ d(v_{1},v_{j})>2$ and $d(v_{2},v_{j})> 2,$ then by Lemma $2.4$, $v_{j}$ has no common neighbours with either $v_{1}$ or $v_{2}.$

From Lemma $2.7$ and Lemma $2.8$ for $n\geq 5,$ we need not to consider the case $d(v_{1},v_{j})=1$ and $d(v_{2},v_{j})=2.$

Thus, for every $v_{j}\in V^{'}, j=3,4,\cdots g$ there are at most two vertices that are common neighbours to $v_{j}$ and either with $v_{1}$ or $v_{2}$. This implies that $|PN(v_{1})|+|PN(v_{2})|\geq 2(n-1)-2(g-2)= 2(n-g+1).$ Hence, either $|PN(v_{1})|\geq n-g+1 $ or $|PN(v_{2})|\geq n-g+1$ or both of them hold.

{Subcase 2.2}. For each $v_{i}, v_{j} \in V^{'} (i\neq j)$ and $d(v_{i},v_{j})> 2.$

Choose every pair of vertices in $V^{'}$, they have no common neighbour. Hence, $ |PN(v_{1})|= |PN(v_{2})|=n+1 > n-g+1 .$

So the result holds. Because of the symmetric of the folded hypercube, it is possible to change the role of $v_{1}$ and $v_{2}$ in the following discussions.
\end{proof}

For convenience, let $g(n+1)-\frac{1}{2}g(g+1)+1=f_{n}(g)$. Combing Lemma $2.5$, Lemma $2.6$ and Lemma $2.7$, we can determine the minimum neighbour of $FQ_{n}$.

\begin{thm} Let $FQ_{n}$ be a folded hypercube of dimension $n,$ $n\geq 5$, and $V^{'}$ be any vertex subset of $V(FQ_{n})$ with $|V^{'}|=g,$ $1\leq g \leq n+2.$ Then $\theta _{FQ_{n}}(g)= f_{n}(g).$ Moreover, let $V^{'}$ be a vertex subset in $FQ_{n}$ that consists of a vertex and its $g-1$ adjacent vertices, then $|N_{FQ_{n}}(V^{'})|= f_{n}(g)$ for $1\leq g \leq n+2.$
\end{thm}

\begin{proof} We prove this result by induction on $g$.

First, we prove the result for $g=1$ and $g=2$. As every vertex in $FQ_{n}$ has $n+1$ neighbours, so $|N_{FQ_{n}}(V^{'})|\geq n+1$ holds for $|V^{'}|=1$. By Lemma $2.5,$ the result holds for $g=2.$

Now assume the result is true for $g,$ $g\leq n,$ that is, for any vertex subset $V^{'}$ of $V(FQ_{n})$ with $|V^{'}|=g, |N_{FQ_{n}}(V^{'})|\geq f_{n}(g).$  Then we show that it is also true for $g+1.$ We prove the result by contradiction. Let $V^{'}$ be a vertex subset of $V(FQ_{n})$ with $|V^{'}|=g+1$ such that $|N_{FQ_{n}}(V^{'})|< f_{n}(g+1).$ By Lemma $2.9$, there exists at least one vertex $v\in V^{'}$ such that one of the following hold:

{Case 1.} $v$ has a neighbour in $V^{'}$ and $|PN(v)|\geq n-(g+1)+2.$

Let $V^{''}=V^{'}-\{v\}.$ Then $N_{FQ_{n}}(V^{''})= (N_{FQ_{n}}(V^{'})- PN(v))\bigcup \{v\},$ and therefore,
$|N_{FQ_{n}}(V^{''})|= |N_{FQ_{n}}(V^{'})|- |PN(v)|+1 < f_{n}(g+1)-[n-(g+1)+2]+1= f_{n}(g). $ A contradiction.

{Case 2.} $v$ has no neighbour in $V^{'}$ and $|PN(v)|\geq n-(g+1)+1.$

Let $V^{''}=V^{'}-\{v\}.$ Then $N_{FQ_{n}}(V^{''})= (N_{FQ_{n}}(V^{'})- PN(v))$ and therefore,
$|N_{FQ_{n}}(V^{''})|= |N_{FQ_{n}}(V^{'})|- |PN(v)| < f_{n}(g+1)-[n-(g+1)+1] = f_{n}(g). $ A contradiction.

So $|N_{FQ_{n}}(V^{'})|\geq f_{n}(g),$  where $V^{'}$ be any vertex subset of $V(FQ_{n})$ with $|V^{'}|=g$ and $1\leq g\leq n+1.$

Following, we show the result is true for $g=n+2.$
Let $V^{'}$ be any vertex subset of $V(FQ_{n})$ with $|V^{'}|=n+2$. Choose a node $v\in V^{'}$ and let $V^{''}=V^{'}-\{v\}$, then $N_{FQ_{n}}(V^{'})\supseteq N_{FQ_{n}}(V^{'}-\{v\})-\{v\}$, that is $|N_{FQ_{n}}(V^{'})|\geq |N_{FQ_{n}}(V^{''})|-1 \geq f_{n}(n+1)-1=f_{n}(n+2).$

So combining all the cases, $|N_{FQ_{n}}(V^{'})|\geq f_{n}(g).$

Let $V^{'}$ be a vertex subset in $FQ_{n}$ that consists of a vertex and its $g-1$ adjacent vertices, by Lemma $2.3$, $|N_{FQ_{n}}(V^{'})|=g(n+1)-2(g-1)-\frac{1}{2}(g-1)(g-2)=f_{n}(g).$
\end{proof}

By simple calculation, $f_{n}(g)=g(n+1)-\frac{1}{2}g(g+1)+1$ is strictly monotonically increasing when $1\leq g \leq n$. Moreover, the maximum of $f_{n}(g)$ is $f_{n}(n)=f_{n}(n+1)= \frac{1}{2}n(n+1)+1$, and $f_{n}(n)=f_{n}(n+1)> f_{n}(n-1)=f_{n}(n+2) >f_{n}(g)$ for $1\leq g \leq n-2.$ $(*)$

\begin{thm}
 Let $ n\geq 8$ and $1\leq g\leq n-1.$ Then $c\kappa_{g+1}(FQ_{n})=f_{n}(g).$
\end{thm}

\begin{proof} Let $F$ be a minimum $(g+1)$-component cut of $FQ_{n}$, then $|F|=c\kappa_{g+1}(FQ_{n})$ and $FQ_{n}-F$ induces at least $g+1$ components. First we prove $|F|\leq f_{n}(g)$ for $1\leq g\leq n+1$.
Let $v\in V(FQ_{n}),$ $S=\{v_{1}, v_{2}, \cdots, v_{g}| v_{i}\in N(v), 1\leq i \leq g\}$. Then $N_{FQ_{n}}(S)$ is a $(g+1)$-component cut of $FQ_{n}$ and at least $g$ of them are singletons. So $c\kappa_{g+1}(FQ_{n})\leq f_{n}(g)$ for $1\leq g \leq n+1.$

Next, we prove that $c\kappa_{g+1}(FQ_{n})\geq f_{n}(g)$ for $1\leq g \leq n+1.$
Suppose the contrary, that is, $c\kappa_{g+1}(FQ_{n})\leq f_{n}(g)-1,$ then $|F|\leq f_{n}(g)-1.$
 By $*,$ we have $|F|\leq \leq f_{n}(g)-1 \leq f_{n}(n-1)-1 = n(n+1)/2-1< \theta_{Q_{n}}(n+3)=n(n+3)/2-7$ for $n\geq 7.$ By Lemma $2.2,$ $Q_{n}-F$ contains exactly one large component $C$ of order at least $2^{n}-|F|-(n+2)$. Noting that $Q_{n}$ is a spanning subgraph of $FQ_{n},$ so $C$ is a connected subgraph in $FQ_{n}$. Let the components of $FQ_{n}-F$ be $C_{1}, C_{2}, C_{3}, \cdots , C_{l} (l\geq g+1),$ and assume that the largest component of $FQ_{n}-F$ is $C_{1}$. Then $g\leq|\bigcup_{i=2}^{l}V(C_{i})|\leq n+2.$ As we have $N_{FQ_{n}}(\bigcup_{i=2}^{l}V(C_{i}))\subseteq F,$ by Theorem $2.10,$ we have $|F|\geq |N_{FQ_{n}}(\bigcup_{i=2}^{l}V(C_{i}))|\geq f_{n}(g)$, a contradiction.

 {\bf Case 2.} $g=n, n+1.$

By $(*),$ we have $f_{n}(n)=f_{n}(n+1)= n(n+1)/2+1.$ If we have $c\kappa_{n+1}(FQ_{n})\geq f_{n}(n)$, then from $f_{n}(n+1)\geq c\kappa_{n+2}(FQ_{n}) \geq c\kappa_{n+1}(FQ_{n}) \geq f_{n}(n),$ we have $c\kappa_{n+2}(FQ_{n})= f_{n}(n+1).$ So we just need to show that $c\kappa_{n+1}(FQ_{n}) \geq f_{n}(n)$. Suppose to the contrary, that is, $c\kappa_{n+1}(FQ_{n})\leq f_{n}(n)-1.$ Then $|F|\leq f_{n}(n)-1=n(n+1)/2 < \theta_{Q_{n}}(n+3)=n(n+3)/2-7$ for $n\geq 8.$ So with the same discuss as to $1\leq g \leq n-1,$ we have $|F|\geq f_{n}(n).$

So $c\kappa_{g+1}(FQ_{n})\geq f_{n}(g)$ for $1\leq g \leq n+1.$

\end{proof}

One may wonder why we only do the result for $1\leq g \leq n+1$ and not extend the result for $g=n+2.$ The reason is that the formula does not hold for $g=n+2.$ And as proved in the following result.

\begin{lem} Let $n\geq 5,$ then $c\kappa_{n+3}(FQ_{n}) > f_{n}(n+2).$

\end{lem}
\begin{proof} Suppose to the contrary, that is, $c\kappa_{n+3}(FQ_{n}) \leq f_{n}(n+2).$ Since $f_{n}(n+2)< f_{n}(n+1) \leq c\kappa_{n+1}(FQ_{n})$, this implies that $c\kappa_{n+3}(FQ_{n})< c\kappa_{n+1}(FQ_{n}).$ A contradiction.
\end{proof}

\section{Conclusions}
The component connectivity is a generalization of standard connectivity of graphs, see \cite{Chartrand,Sampathkumar}, which can be viewed as a measure of robustness of interconnection networks. The standard connectivity of hypercubes (or classic networks) have been studied by many authors, but there are few papers on the component connectivity of networks. As the folded hypercube is an enhancement on the hypercube $Q_{n}$, they have some similar properties. Motivated by the method  in \cite{Hsu}. We introduce an idea to consider the  $g$-component connectivity of the folded hypercube for $ 2\leq g\leq n+2$ for $n\geq 8$. Our result in this note is not complete.
 The problem of determining the $(g+1)$-component connectivity of the folded hypercube for $g\geq n+2$ is still open.

\section{Acknowledgements}

The research is supported by NSFC (No.11301371,  61502330), SRF for ROCS, SEM and Natural Sciences
Foundation of Shanxi Province (No. 2014021010-2), Fund Program for the Scientific Activities of Selected
Returned Overseas Professionals in Shanxi Province.

\begin{rem}
The work was included in the MS thesis of the first author in [On the component connectiviy of hypercubes and folded hypercubes, MS Thesis at Taiyuan University of
Technology, 2017]. 
\end{rem}

\end{document}